\documentclass{article}
\usepackage{amssymb}
\usepackage{amsfonts}
\usepackage{amsmath}

\newtheorem{theorem}{Theorem}
\newtheorem{acknowledgement}[theorem]{Acknowledgement}

\newtheorem{corollary}[theorem]{Corollary}

\newtheorem{definition}[theorem]{Definition}

\newtheorem{lemma}[theorem]{Lemma}
\newtheorem{notation}[theorem]{Notation}

\newtheorem{proposition}[theorem]{Proposition}
\newtheorem{remark}[theorem]{Remark}

\newenvironment{proof}[1][Proof]{\noindent\textbf{#1.} }{\ \rule{0.5em}{0.5em}}
\input{tcilatex}
\begin{document}

\title{Homogenization of linear parabolic equations with a certain resonant
matching between rapid spatial and temporal oscillations in periodically
perforated domains}
\author{Tatiana Lobkova \\
Department of Quality Technology and Management, Mechanical\\
Engineering and Mathematics, Mid Sweden University, \\
S-83125 \"{O}stersund Sweden}
\date{}
\maketitle

\begin{abstract}
In this article, we study homogenization of a parabolic linear problem
governed by a coefficient matrix with rapid spatial and temporal
oscillations in periodically perforated domains with homogeneous Neumann
data on the boundary of the holes. We prove results adapted to the problem
for a characterization of multiscale limits for gradients and very weak
multiscale convergence.

Keywords: Homogenization, two-scale convergence, multiscale convergence,
periodically perforated domains

MSC 2000: 35B27; 35K10
\end{abstract}

\section{Introduction}

Homogenization theory deals with the question of finding effective
properties and microvariations in heterogeneous materials. However, it is
difficult to handle the rapid periodic oscillations of coefficients that
govern partial differential equations describing processes in such
materials. In this paper, we study a parabolic problem with a certain
resonant matching between rapid oscillations in space and time in
periodically perforated domains with homogeneous Neumann data on the
boundary of the holes. Homogenization for linear parabolic problems with
rapid oscillations of similar kind as in this paper was achived already by
Bensoussan, Lions and Papanicolaou in \cite{BLP} using asymptotic expansions
for domains without perforations. See also the pioneering work \cite{ColSpa}
from 1977 by Colombini and Spagnolo, where homogenization of linear
parabolic equations with rapid spatial oscillations is performed. A further
development of parabolic homogenization problems applying techniques of
two-scale convergence type was presented by Holmbom in 1997, see \cite{Ho},
where the first compactness result of very weak multiscale convergence type
was shown for one rapid scale in space and time each. A similar result with
the setting of $\Sigma $-convergence was obtained in 2007 by Nguetseng and
Woukeng in \cite{NgWo}. Later in \cite{FHOP1} from 2010 was proven a
compactness result for the case with $n$ well-separated spatial scales by
Flod\'{e}n et. al.. Multiscale convergence techniques for linear parabolic
problems for two rapid time scales with one of them identical to the single
rapid spatial scale were achived by Flod\'{e}n and Olsson in 2007, see \cite%
{FlOl}. In 2009, these results were extended by Woukeng, who studied
non-linear parabolic problems with the same choice of scales in \cite{Wo}.
Also \cite{Per1},\ by Persson, deals with monotone parabolic problems, but
with an arbitrary number of temporal microscales, where none of them has to
be identical with the rapid spatial scale or even has to be a power of $%
\varepsilon .$ In \cite{FHOLP} we return to the case of linear parabolic
homogenization for arbitrary numbers of spatial and temporal scales
benefitting from the concept of jointly separated scales introduced in \cite%
{Per3}.

Perforated domain means facing a further difficulty of a different kind than
for oscillating coefficients. Let $\Omega $ be a bounded domain in $%
\mathbb{R}
^{N}$ and $\Omega _{\varepsilon }$ a corresponding perforated domain with
small holes which occur with a period $\varepsilon .$ Advanced extension
techniques to maintain a priori estimates in perforated domains for a linear
elliptic problem with Neumann data on the boundary of the perforations is
used by Cioranescu and Saint Jean Paulin in \cite{SrSj}. Examples of further
developed extension techniques for a larger class of perforated domains is
presented by Acerbi et. al. in \cite{ACDP}. In \cite{AlMr} and \cite{Al1}
Allaire developed methods which are independent of advanced extension
techniques. \cite{Al1} introduces methods of two-scale convergence type and
these methods we adapt to the time-dependent problem in the present paper.
With inhomogeneous Neumann data on the boundary of the holes the problem
becomes more complicated as the assumptions must be adapted to the fact that
the area of interface between holes and domain increases when $\varepsilon $
decreases, see \cite{CDDGZ}, \cite{CoDo}.

An early study of evolution problems in perforated domains is found in \cite%
{DN}, where a parabolic problem with fast oscillations in one spatial scale
is studied using advanced \ extension techniques and in a slightly more
general setting in \cite{DN1}. In 2016 Donato and Yang \cite{DoYa} performed
a generalization of \cite{DN} by using the time-dependent unfolding method
adapted to perforated domains. \cite{DN}, \cite{DN1} and \cite{DoYa} deal
with homogenous Neumann data on the boundary of the holes. See also \cite%
{AiNa}. In e.g. \cite{PiRy} the case of non-homogenous Neumann data is
studied for a nonlinear parabolic problem with oscillations in one rapid
spatial scale.

In this paper we study homogenization of the parabolic linear problem with
spatial and temporal oscillations%
\begin{eqnarray}
\partial _{t}u_{\varepsilon }(x,t)-\nabla \cdot \left( A\left( \frac{x}{%
\varepsilon },\frac{t}{\varepsilon ^{2}}\right) \nabla u_{\varepsilon
}(x,t)\right) &=&f_{\varepsilon }(x,t)\text{ in }\Omega _{\varepsilon
}\times (0,T)\text{,}  \notag \\
u_{\varepsilon }(x,t) &=&0\text{ on }\partial \Omega \times (0,T)\text{,}
\label{introdequat1} \\
A\left( \frac{x}{\varepsilon },\frac{t}{\varepsilon ^{2}}\right) \nabla
u_{\varepsilon }(x,t)\cdot n &=&0\text{ on }(\partial \Omega _{\varepsilon
}-\partial \Omega )\times (0,T)\text{,}  \notag \\
u_{\varepsilon }(x,0) &=&u_{\varepsilon }^{0}(x)\text{ in }\Omega
_{\varepsilon }\text{,}  \notag
\end{eqnarray}%
where$\ f_{\varepsilon }\in L^{2}(\Omega _{\varepsilon }\times (0,T))$ and $%
u_{\varepsilon }^{0}\in L^{2}(\Omega _{\varepsilon })$. We develop a method
without nontrivial extensions that generalize the approach in \cite{Al1} and
bring forth a gradient characterization adapted to the problem. In
particular, we show a result of very weak multiscale convergence type and
perform the homogenization procedure for the problem. \cite{DoWou} is
sharing similarities with our problem but includes methods that are based on
nontrivial extension techniques. See also \cite{DoTe}.

The paper is organized as follows. In Section 2 we give a brief description
of two-scale convergence and its generalization to cases with larger numbers
of scales, see Definition \ref{Definition2} and Definition \ref{Definition11}%
. We take a look at the recently developed idea of very weak multiscale
convergence, see Definition \ref{Definition13}. Section 3 is dedicated to
the multiscale convergence for sequences of time-dependent functions in
perforated domains and the answering concept for very weak multiscale
convergence with one rapid scale in space and time each. We define
perforated domains $\Omega _{\varepsilon }.$ Then we prove in Proposition %
\ref{P2,2} an essential result for the regularity of the $(2,2)$-scale limit
for bounded sequences in $L^{2}(0,T;H^{1}(\Omega _{\varepsilon }))$. In
Theorem \ref{first in line} we find a characterization of the (2,2)-scale
limit for $\left\{ \nabla u_{\varepsilon }\right\} $ under certain
assumptions and in Corollary \ref{Corr1} we consider a version of very-weak
multiscale convergence for the same choice of scales. Finally, in Section 4
we state a homogenization result which is proven by applying the results
from Section 3.

\begin{notation}
Below we introduce a list of sets and function spaces.
\end{notation}

$\Omega :$ Open bounded subset of $%
\mathbb{R}
^{N}$ with a smooth boundary.

$\Omega _{\varepsilon }:$ A domain with small holes situated periodically in 
$\Omega $. See more about

notation for perforated domains in Section 3.

$Y:$ The unit cube $(0,1)^{N}$.

$Y^{\ast }:$ An open subset of $Y$.

$E_{\sharp }(Y^{\ast }):$ The $Y^{\ast }$ periodic extension of $Y^{\ast }$
infinitely along all principal

directions of $%
\mathbb{R}
^{N}$.

$S:$ The intervall $(0,1)$.

$\mathcal{Y}_{n,m}:$ The set $Y^{n}\times S^{m}$.

$A,B:$ Subsets of $%
\mathbb{R}
^{N}$.

$G(A):$ A space of real valued functions defined on $A$.

$G(A)/%
\mathbb{R}
:$ The space of functions $\left\{ u\in G(A)\mid \int_{A}u(y)dy=0\right\} $.

$D(\Omega ):$ The space of $C^{\infty }(\Omega )$ - functions with compact
support in $\Omega $.

$C_{\sharp }^{\infty }(Y):$ The space of $Y$-periodic functions in $%
C^{\infty }(%
\mathbb{R}
^{N})$.

$H_{\sharp }^{1}(Y):$\ The closure of $C_{\sharp }^{\infty }(Y)$ with
respect to the $H^{1}(Y)-$ norm.

$C_{\sharp }^{\infty }(Y^{\ast }):$ $C^{\infty }(E_{\sharp }(Y^{\ast }))$ -
functions that are periodic with respect to $Y^{\ast }$.

$D_{\sharp }(Y^{\ast }):$ The functions in $C_{\sharp }^{\infty }(Y^{\ast })$
with support contained in $E_{\sharp }(Y^{\ast })$.

$H_{\sharp }^{1}(Y^{\ast }):$ $H_{loc}^{1}(E_{\sharp }(Y^{\ast }))$ -
functions that are periodic with respect to $Y^{\ast }$.

$L^{2}(a,b;G(A)):$ The set of functions $\left\{ u:(a,b)\rightarrow G(A)\mid
\int_{a}^{b}\left\Vert u\right\Vert _{G(A)}^{2}dt<\infty \right\} $.

$D(B;G(A)):$ The space of infinitely differentiable functions $\left\{ u\mid
u:B\rightarrow G(A)\right\} $

with compact support in $B$.

\section{The two-scale convergence}

The concept was originally introduced by Nguetseng \cite{Ng} and later in
90's further developed by Allaire \cite{Al1}. A review of classical
two-scale convergence from 2002 can be found in \cite{LNW}. Yet it works for
more than two scales, see \cite{AlBr}. A quite attractive generalization of
two-scale convergence, scale convergence, is introduced by Mascarenhas and
Toader in \cite{MaTo}. Moreover, \cite{Si} adapted ideas of
scale-convergence from \cite{MaTo} and from a different kind of
generalization of two-scale convergence in \cite{HSSW} to develop the
concept of $\lambda $-scale convergence. See also \cite{Per2}. Moreover,
Nguetseng also introduced a quite sophisticated concept, $\sum $%
-convergence, which goes beyond the periodic setting, see e.g. \cite{NgSv}.
Another important improvement in the two-scale convergence theory was made
by Pak \cite{Pa} in 2005. He adapted it to differential forms and manifolds.
We would also like to mention \cite{CDG}.

Let us begin with the classical definition by Nguetseng and Allaire which
was shown for the case of bounded sequences in $L^{2}$.

\begin{definition}
\label{Definition2}We say that a bounded sequence of functions $\left\{
u_{\varepsilon }\right\} $ in $L^{2}(\Omega )$ two-scale converges to a
limit $u_{0}\in L^{2}(\Omega \times Y)$, if 
\begin{equation*}
\underset{\varepsilon \rightarrow 0}{\lim }\int_{\Omega }u_{\varepsilon
}(x)v\left( x,\frac{x}{\varepsilon }\right) dx=\int_{\Omega
}\int_{Y}u_{0}(x,y)v(x,y)dydx
\end{equation*}%
holds for all $v\in D(\Omega ;C_{\sharp }^{\infty }(Y))$. We write $%
u_{\varepsilon }\overset{}{\overset{2}{\underset{}{\rightharpoonup }}u_{0}}.$
If, in addition,%
\begin{equation*}
\underset{\varepsilon \rightarrow 0}{\lim }\left\Vert u_{\varepsilon
}(x)-u_{0}\left( x,\frac{x}{\varepsilon }\right) \right\Vert _{L^{2}(\Omega
)}=0\text{,}
\end{equation*}%
we say that $\left\{ u_{\varepsilon }\right\} $ two-scale converges strongly
to $u_{0}$.
\end{definition}

\begin{remark}
The strong two-scale convergence is also called a corrector type result,
according to the vocabulary of homogenization.
\end{remark}

\begin{theorem}
Two-scale limits are unique.
\end{theorem}

\begin{proof}
See reasoning after Definition 1 in \cite{LNW}.
\end{proof}

\begin{theorem}
\label{compactness T} Let $\left\{ u_{\varepsilon }\right\} $ be a bounded
sequence in $L^{2}(\Omega ).$ Then $\left\{ u_{\varepsilon }\right\} $ is
compact with respect to the two-scale convergence, i.e. there exist a
subsequence $\left\{ u_{\varepsilon ^{\prime }}\right\} $ two-scale
converging to a function $u_{0}\in L^{2}(\Omega \times Y).$
\end{theorem}

\begin{proof}
See proof of Theorem 4.1 in \cite{Ne}.
\end{proof}

The next theorem shows relations between norms for weak $L^{2}(\Omega )$%
-limits and two-scale limits.

\begin{theorem}
Let $\left\{ u_{\varepsilon }\right\} $ be a sequence in $L^{2}(\Omega )$
that two-scale converges to\newline
$u_{0}\in L^{2}(\Omega \times Y).$ Then%
\begin{equation*}
\lim_{\varepsilon \rightarrow 0}\inf \left\Vert u_{\varepsilon }\right\Vert
_{L^{2}(\Omega )}\geq \left\Vert u_{0}\right\Vert _{_{L^{2}(\Omega \times
Y)}}\geq \left\Vert u\right\Vert _{L^{2}(\Omega )}\text{,}
\end{equation*}%
where 
\begin{equation*}
u_{\varepsilon }(x)\rightharpoonup u(x)=\int_{Y}u_{0}(x,y)dy\text{ in }%
L^{2}(\Omega )\text{.}
\end{equation*}
\end{theorem}

\begin{proof}
See Theorem 10 in \cite{LNW}.
\end{proof}

Phenomenons of two-scale \ convergence type may appear under certain
conditions also when neither of the involved sequences originates from an
admissible test function.

\begin{theorem}
\label{important} Let $\left\{ u_{\varepsilon }\right\} $ be a bounded
sequence in $L^{2}(\Omega )$ which two-scale converges to $u_{0}\in
L^{2}(\Omega \times Y)$ and assume that 
\begin{equation*}
\lim_{\varepsilon \rightarrow 0}\left\Vert u_{\varepsilon }(x)\right\Vert
_{L^{2}(\Omega )}=\left\Vert u_{0}(x,y)\right\Vert _{L^{2}(\Omega \times Y)}%
\text{.}
\end{equation*}%
Then%
\begin{equation*}
\lim_{\varepsilon \rightarrow 0}\int_{\Omega }u_{\varepsilon
}(x)v_{\varepsilon }(x)\tau (x)dx=\int_{\Omega }\int_{Y}u_{0}(x,y)v(x,y)\tau
(x)dydx
\end{equation*}%
for any $\tau \in D(\Omega )$.

\noindent Moreover, if the Y-periodic extension of $u_{0}$ belongs to $%
L^{2}(\Omega ;C_{\sharp }(Y))$, then%
\begin{equation*}
\lim_{\varepsilon \rightarrow 0}\left\Vert u_{\varepsilon }(x)-u_{0}\left( x,%
\frac{x}{\varepsilon }\right) \right\Vert _{L^{2}(\Omega )}=0\text{.}
\end{equation*}
\end{theorem}

\begin{proof}
See Theorem 11 in \cite{LNW}.
\end{proof}

The two-scale convergence can be used in many applications due to the
compactness property.

\begin{remark}
In \cite{Al1} Allaire demonstrates proof of Theorem \ref{compactness T} for
the test functions from $L^{2}(\Omega ;C_{\sharp }(Y))$.
\end{remark}

We also define the concept of very weak two-scale convergence using a
smaller class of test functions than usual two-scale convergence.

\begin{definition}
Let $\left\{ \varphi _{\varepsilon }\right\} $ be a sequence of functions in 
$L^{1}(\Omega )$. We say that $\left\{ \varphi _{\varepsilon }\right\} $
two-scale convergences very weakly to $\varphi _{1}\in L^{1}(\Omega \times
Y) $ if%
\begin{equation*}
\underset{\varepsilon \rightarrow 0}{\lim }\int_{\Omega }\varphi
_{\varepsilon }(x)v_{1}(x)v_{2}\left( \frac{x}{\varepsilon }\right)
dx=\int_{\Omega }\int_{Y}\varphi _{1}(x,y)v_{1}(x)v_{2}\left( y\right) dydx
\end{equation*}%
for all $v_{1}\in D(\Omega )$ and all $v_{2}\in C_{\sharp }^{\infty }(Y)/%
\mathbb{R}
$ and 
\begin{equation*}
\int_{Y}\varphi _{1}(x,y)dy=0\text{.}
\end{equation*}%
We write
\end{definition}

\begin{equation*}
\varphi _{\varepsilon }(x)\overset{2}{\underset{vw}{\rightharpoonup }}%
\varphi _{1}(x,y)\text{.}
\end{equation*}%
For $\left\{ u_{\varepsilon }\right\} $ bounded in $H^{1}(\Omega )$ there is
a characterization of the two-scale limit for $\left\{ \nabla u_{\varepsilon
}\right\} $ which means that, up to a subsequence,%
\begin{equation*}
\nabla u_{\varepsilon }(x)\overset{2}{\rightharpoonup }\nabla u(x)+\nabla
_{y}u_{1}(x,y)\text{,}
\end{equation*}%
where $u$ is the weak $H^{1}(\Omega )$-limit for $\left\{ u_{\varepsilon
}\right\} $ and $u_{1}\in L^{2}(\Omega ,H_{\sharp }^{1}(Y)/%
\mathbb{R}
)$.

In fact, there is a connection between the very weak two-scale limit for $%
\left\{ \varepsilon ^{-1}u_{\varepsilon }\right\} $ and the two-scale limit
for $\left\{ \nabla u_{\varepsilon }\right\} $. It is possible to establish
a compactness result for a sequence $\left\{ \varepsilon ^{-1}u_{\varepsilon
}\right\} ,$ though it is not bounded in any Lebesgue space. We have that%
\begin{equation*}
\frac{u_{\varepsilon }}{\varepsilon }\overset{2}{\underset{vw}{%
\rightharpoonup }}u_{1}\text{,}
\end{equation*}%
up to a subsequence, if $\left\{ u_{\varepsilon }\right\} $ is bounded in $%
H_{0}^{1}(\Omega )$.

Just as with regular two-scale convergence we can generalize this result to
be valid for several scales and to the evolution setting. See e.g. \cite%
{FHOLP}.

For regular multiscale convergence, we have certain assumptions about how
the scales relate to each other. Assuming that the scales in the lists $%
\left\{ \varepsilon _{1},\ldots ,\varepsilon _{n}\right\} $ and $\left\{
\varepsilon _{1}^{\prime },\ldots ,\varepsilon _{m}^{\prime }\right\} $ are
microscopic, i.e. $\varepsilon _{k},$ $\varepsilon _{k}^{\prime }$ goes to
zero when $\varepsilon $ does. We say, according to the definition in \cite%
{AlBr}, that the scales in one list are separated if%
\begin{equation*}
\lim_{\varepsilon \rightarrow 0}\frac{\varepsilon _{k+1}}{\varepsilon _{k}}=0
\end{equation*}%
and well-separated if there exists a positive integer $\ell $ such that%
\begin{equation*}
\lim_{\varepsilon \rightarrow 0}\frac{1}{\varepsilon _{k}}\left( \frac{%
\varepsilon _{k+1}}{\varepsilon _{k}}\right) ^{\ell }=0
\end{equation*}%
where $k=1,\ldots ,n-1$. For the evolution setting we need the equivalent
for multiscale convergence with time-dependent effect. Following \cite{Per1}
we provide the concept in the next definition.

\begin{definition}
Let $\left\{ \varepsilon _{1},\ldots ,\varepsilon _{n}\right\} $ and $%
\left\{ \varepsilon _{1}^{\prime },\ldots ,\varepsilon _{m}^{\prime
}\right\} $ be lists of well-separated scales (see \cite{AlBr}). Collect all
elements from both lists in one common list. If from possible duplicates,
where by duplicates we mean scales which tend to zero equally fast, one
member of each such pair is removed and the list in order of magnitude of
all the remaining elements is well-separated, the lists $\left\{ \varepsilon
_{1},\ldots ,\varepsilon _{n}\right\} $ and $\left\{ \varepsilon
_{1}^{\prime },\ldots ,\varepsilon _{m}^{\prime }\right\} $ are said to be
jointly well-separated. Moreover, $y^{n}=(y_{1},y_{2},\ldots ,y_{n})$ and $%
s^{m}=(s_{1},s_{2},\ldots ,s_{m})$.
\end{definition}

\begin{remark}
For some more examples and a technically strict definition, see Section 2.4
in \cite{Per1}.
\end{remark}

Below we provide a characterization of multiscale limits for gradients.

\begin{definition}
\label{Definition11}A sequence $\left\{ u_{\varepsilon }\right\} $ in $%
L^{2}(\Omega \times (0,T))$ is said to $\left( n+1,m+1\right) $- scale
converge to $u_{0}\in L^{2}(\Omega \times (0,T)\times \mathcal{Y}_{n,m})$ if%
\begin{gather*}
\lim_{\varepsilon \rightarrow 0}\int_{0}^{T}\int_{\Omega }u_{\varepsilon
}(x,t)v\left( x,t,\frac{x}{\varepsilon _{1}},\cdots ,\frac{x}{\varepsilon
_{n}},\frac{t}{\varepsilon _{1}^{\prime }},\cdots ,\frac{t}{\varepsilon
_{m}^{\prime }}\right) dxdt \\
=\int_{0}^{T}\int_{\Omega }\int_{\mathcal{Y}%
_{n,m}}u_{0}(x,t,y^{n},s^{m})v(x,t,y^{n},s^{m})dy^{n}ds^{m}dxdt
\end{gather*}%
for all $v\in L^{2}(\Omega \times (0,T);C_{\sharp }(\mathcal{Y}_{n,m})).$
This is denoted by 
\begin{equation*}
u_{\varepsilon }(x,t)\overset{n+1,m+1}{\rightharpoonup }%
u_{0}(x,t,y^{n},s^{m})\text{.}
\end{equation*}
\end{definition}

We give the compactness result for $\left( n+1,m+1\right) $-scale
convergence.

\begin{theorem}
\label{compactness result}Let $\left\{ u^{\varepsilon }\right\} $ be a
bounded sequence in $L^{2}(\Omega _{T})$ and assume that the lists $\left\{
\varepsilon _{1},\ldots ,\varepsilon _{n}\right\} $ and $\left\{ \varepsilon
_{1}^{\prime },\ldots ,\varepsilon _{m}^{\prime }\right\} $ are jointly
separated. Then there exists a function $u_{0}\in L^{2}(\Omega _{T}\times 
\mathcal{Y}_{n,m})$ such that%
\begin{equation*}
u^{\varepsilon }(x,t)\overset{n+1,m+1}{\rightharpoonup }%
u_{0}(x,t,y^{n},s^{m})\text{,}
\end{equation*}%
up to a subsequence.
\end{theorem}

\begin{proof}
See Theorem 17 in \cite{FHOLP} and also Theorem 2.66 in \cite{Per1}%
.\smallskip
\end{proof}

\begin{theorem}
\label{gradients} Let $\left\{ u_{\varepsilon }\right\} $ be a bounded
sequence in $W_{2}^{1}(0,T;H_{0}^{1}(\Omega ),L^{2}\left( \Omega \right) )$
and suppose that the lists $\left\{ \varepsilon _{1},\ldots ,\varepsilon
_{n}\right\} $ and $\left\{ \varepsilon _{1}^{\prime },\ldots ,\varepsilon
_{m}^{\prime }\right\} $ are jointly well-separated. Then there exists a
subsequence such that%
\begin{eqnarray*}
u_{\varepsilon }(x,t) &\rightarrow &u(x,t)\text{ in }L^{2}(\Omega \times
(0,T))\text{,} \\
u_{\varepsilon }(x,t) &\rightharpoonup &u(x,t)\text{ in }%
L^{2}(0,T;H_{0}^{1}(\Omega ))
\end{eqnarray*}%
and%
\begin{equation*}
\nabla u_{\varepsilon }(x,t)\overset{n+1,m+1}{\rightharpoonup }\nabla
u(x,t)+\sum_{k=1}^{n}\nabla _{y_{k}}u_{k}(x,t,y^{k},s^{m})\text{,}
\end{equation*}%
where $u\in W_{2}^{1}(0,T;H_{0}^{1}(\Omega ),L^{2}(\Omega ))$, $u_{1}\in
L^{2}(\Omega \times (0,T)\times S^{m};H_{\sharp }^{1}(Y_{1})/%
\mathbb{R}
)$ and $u_{k}\in L^{2}(\Omega \times (0,T)\times \mathcal{Y}%
_{k-1,m};H_{\sharp }^{1}(Y_{k})/%
\mathbb{R}
)$ for $k=2,\ldots ,n$.
\end{theorem}

\begin{proof}
See Theorem 2.74 in \cite{Per1} or the Appendix of \cite{FHOLP}.
\end{proof}

We define very weak evolution multiscale convergence.

\begin{definition}
\label{Definition13}A sequence $\left\{ \varphi _{\varepsilon }\right\} $ in 
$L^{1}(\Omega \times (0,T))$ is said to $\left( n+1,m+1\right) $-scale
converge very weakly to $\varphi _{n}\in L^{1}(\Omega \times (0,T)\times 
\mathcal{Y}_{n,m})$ if%
\begin{eqnarray*}
&&\int_{0}^{T}\int_{\Omega }\varphi _{\varepsilon }(x,t)v_{1}\left( x,\frac{x%
}{\varepsilon _{1}},\cdots ,\frac{x}{\varepsilon _{n-1}}\right) v_{2}\left( 
\frac{x}{\varepsilon _{n}}\right) c\left( t,\frac{t}{\varepsilon
_{1}^{\prime }},\cdots ,\frac{t}{\varepsilon _{m}^{\prime }}\right) dxdt \\
&\rightarrow &\int_{0}^{T}\int_{\Omega }\int_{\mathcal{Y}_{n,m}}\varphi
_{n}(x,t,y^{n},s^{m})v_{1}(x,y^{n-1})v_{2}(y_{n})c(t,s^{m})dy^{n}ds^{m}dxdt
\end{eqnarray*}%
for all $v_{1}\in D(\Omega ;C_{\sharp }^{\infty }(Y^{n-1}))$, $v_{2}\in
C_{\sharp }^{\infty }(Y_{n})/%
\mathbb{R}
$ and $c\in D(0,T;C_{\sharp }^{\infty }(S^{m})).$\newline
A unique limit is provided by requiring that%
\begin{equation*}
\int_{Y_{n}}\varphi _{n}(x,t,y^{n},s^{m})dy_{n}=0\text{.}
\end{equation*}%
We write%
\begin{equation*}
\varphi _{\varepsilon }(x,t)\underset{vw}{\overset{n+1,m+1}{\rightharpoonup }%
}\varphi _{n}(x,t,y^{n},s^{m})\text{.}
\end{equation*}
\end{definition}

\begin{remark}
Note that the decoupling of the function $v_{2}$ governed by fastest spatial
variable from $v_{1}$ depending on the remaining local spatial variables and
the global variable $x$ is important when proving the compactness result in
Theorem \ref{theorem 17} below because $v_{2}$ has to be found by means of a
certain kind of Poisson equation.
\end{remark}

We are now ready to state compactness result for very weak $(n+1,$ $m+1)$%
-scale convergence.

\begin{theorem}
\label{theorem 17}Let $\left\{ u_{\varepsilon }\right\} $ be a bounded
sequence in $W_{{}}^{1,2}(0,T;H_{0}^{1}\left( \Omega \right) ,L^{2}\left(
\Omega \right) )$ and assume that the lists $\left\{ \varepsilon _{1},\ldots
,\varepsilon _{n}\right\} $ and $\left\{ \varepsilon _{1}^{\prime },\ldots
,\varepsilon _{m}^{\prime }\right\} $ are jointly well-separated. Then there
exists a subsequence such that%
\begin{equation*}
\frac{u_{\varepsilon }(x,t)}{\varepsilon _{n}}\overset{n+1,m+1}{\underset{vw}%
{\rightharpoonup }}u_{n}(x,t,y^{n},s^{m})\text{,}
\end{equation*}%
where, for $n=1,$ $u_{1}\in L^{2}(\Omega \times (0,T)\times S^{m};H_{\sharp
}^{1}(Y_{1})/%
\mathbb{R}
)$ and, for $n=2,3,\ldots ,$ \newline
$u_{n}\in L^{2}(\Omega \times (0,T)\times \mathcal{Y}_{n-1,m};H_{\sharp
}^{1}(Y_{n})/%
\mathbb{R}
).$
\end{theorem}

\begin{proof}
See Theorem 8 in \cite{FHOLP}.
\end{proof}

In Sections 3-4 we consider a special case of very weak multiscale
convergence, where the fast spatial scale is $\varepsilon _{1}=\varepsilon $
and the rapid temporal scale is chosen as $\varepsilon _{1}^{\prime
}=\varepsilon ^{r}$, $r>0$, $n=$ $m=1$ and make the necessary modifications
to suit homogenization in perforated domains.

\section{An adaptation to perforated domains}

By knowing that two-scale convergence can handle homogenization problems in
perforated domains, let us define periodically perforated domains $\Omega
_{\varepsilon }$ in a setting suitable for our problem.

We define $\Omega $ as an open bounded domain in $%
\mathbb{R}
^{N}$, $N\geq 2$, with smooth boundary $\partial \Omega $. We choose $%
Y_{i}^{\varepsilon }$, to be disjoint open cubes with side-length $%
\varepsilon $ such that%
\begin{equation*}
\Omega \subset \bigcup_{i=1}^{N(\varepsilon )}Y_{i}^{\varepsilon }\text{.}
\end{equation*}%
We need to define 
\begin{equation*}
\mathcal{A}=\left\{ i\in 
\mathbb{N}
\left\vert Y_{i}^{\varepsilon }\cap \partial \Omega \neq \varnothing \right.
\right\} \text{.}
\end{equation*}%
We let $Y^{H}\subset \subset Y$, where $Y^{H}$ has smooth boundary; $Y^{\ast
}=Y-Y^{H}$ and $Y_{i}^{\varepsilon \ast }$ are miniatures with side-length $%
\varepsilon $ of $Y^{\ast }$ such that $Y_{i}^{\varepsilon \ast }\subset
Y_{i}^{\varepsilon }$. We define%
\begin{equation*}
\hat{\Omega}_{\varepsilon }=\left( \bigcup_{i=1}^{N(\varepsilon
)}Y_{i}^{\varepsilon \ast }\right) \cap \Omega \text{.}
\end{equation*}%
Furthemore, $S_{\varepsilon }$ is defined as 
\begin{equation*}
S_{\varepsilon }=\left( \bigcup_{i\in \mathcal{A}}Y_{i}^{\varepsilon \ast
}\right) \cap \Omega \text{.}
\end{equation*}%
Analogously, we define 
\begin{equation*}
R_{\varepsilon }=\left( \bigcup_{i\in \mathcal{A}}Y_{i}^{\varepsilon
}\right) \cap \Omega \text{.}
\end{equation*}

We let

\begin{equation*}
\Omega _{\varepsilon }=\left( \hat{\Omega}_{\varepsilon }-S_{\varepsilon
}\right) \cup R_{\varepsilon }\text{.}
\end{equation*}%
This means that we can define $\Omega _{\varepsilon }$ as 
\begin{equation*}
\Omega _{\varepsilon }=\left\{ \left\{ x\in \Omega \left\vert \chi _{Y^{\ast
}}(\frac{x}{\varepsilon })=1\right. \right\} -S_{\varepsilon }\right\} \cup
R_{\varepsilon }\text{,}
\end{equation*}%
where $\chi _{Y^{\ast }}$ is the $Y$-periodic repetition of a function
defined on $Y$ that is equal to one on $Y^{\ast }$and zero elsewhere. Hence,
the perforations do not cut $\partial \Omega .$

\begin{proposition}
\label{P2,2}Let $\left\{ u_{\varepsilon }\right\} $ be bounded in $%
L^{2}(0,T;H^{1}(\Omega _{\varepsilon }))$ and assume that%
\begin{equation}
\widetilde{u_{\varepsilon }}(x,t)\overset{2,2}{\rightharpoonup }u(x,t)\chi
_{Y^{\ast }}(y)\text{,}  \label{assumption}
\end{equation}%
where $u\in L^{2}(\Omega \times (0,T))$ and $\widetilde{u_{\varepsilon }}$
is an extension by zero of $u_{\varepsilon }$ from $\Omega _{\varepsilon }$
to $\Omega .$ Then $u\in L^{2}(0,T;H^{1}(\Omega )).$
\end{proposition}

\begin{definition}
\label{Lteta}Let $\left( D^{\ast }\right) ^{N}=D\left[ \Omega ;D_{\sharp
}(Y_{1}^{\ast }\times \ldots \times Y_{n}^{\ast })\right] ^{N}$ be the set
of smooth functions $v:\Omega \times Y_{1}^{\ast }\times \ldots Y_{n}^{\ast
}\rightarrow 
\mathbb{R}
^{N}$ periodic in $(y_{1},\ldots ,y_{n})$ with compact support in $\Omega $
for $x$ and with their support contained in each of $E_{\sharp }(Y_{k}^{\ast
})$ for the respective variable $y_{k},$ $k=1,\ldots ,n.$
\end{definition}

Next Lemma is cited from \cite{AlBr}.

\begin{lemma}
\label{divergence-free condition} For any $k\in \left\{ 1,\ldots ,n\right\}
, $ let $D_{k}^{\ast }$ be the subset of $\left( D^{\ast }\right) ^{N}$
composed of functions satisfying a \textquotedblleft
generalised\textquotedblright\ divergence-free condition, i.e. 
\begin{equation*}
\left\{ 
\begin{array}{c}
D_{n}^{\ast }=\left\{ v\in (D^{\ast })^{N};\text{ }\func{div}%
_{y_{n}}v=0\right\} \text{,} \\ 
D_{k}^{\ast }=\left\{ v\in (D^{\ast })^{N};\text{ }\func{div}_{y_{n}}v=0%
\text{ and }\int_{Y_{j+1}}\ldots \text{ }\int_{Y_{n}}\func{div}_{y_{j}}v=0%
\text{ }\forall k\leqq j\leqq n-1\right\} \text{.}%
\end{array}%
\right.
\end{equation*}

\noindent These spaces have the following property:\smallskip

\noindent (i)\ Any function $\theta \in $ $D\left[ \Omega ;D_{\sharp
}(Y_{1}^{\ast }\times \ldots \times Y_{k}^{\ast })\right] ^{N}$ can be
expressed as the average of a function in $D_{k}^{\ast }$, i.e. there exists 
$v(x,y_{1,}\ldots ,y_{n})\in D_{k}^{\ast }$ such that%
\begin{equation*}
\theta =\int_{Y_{k+1}}\ldots \int_{Y_{n}}v\text{ }dy_{k+1}\ldots dy_{n}\text{%
.}
\end{equation*}%
(ii) Any function $\theta \in D(\Omega )^{N}$ can also be expressed as the
average of a function $v$ in $D_{n}^{\ast }$, such that 
\begin{equation*}
\theta =\int_{Y_{1}}\ldots \int_{Y_{n}}v\text{ }dy_{1}\ldots dy_{n}\text{
and }\left\Vert v\right\Vert _{L^{2}(\Omega \times Y_{1}^{\ast }\times
\ldots \times Y_{n}^{\ast })^{N}}\leqq C\left\Vert \theta \right\Vert
_{L^{2}(\Omega )^{N}}\text{,}
\end{equation*}%
where the constant $C$ is independent of $v$ and $\theta $.
\end{lemma}

\begin{proof}
See proof of Lemma 4.13 (ii) in \cite{AlBr}.\smallskip
\end{proof}

Further we need to state Corollary \ref{en snabb rumsskala} in order to make
a proof of the Proposition \ref{P2,2}. Hence, the next Corollary, which
means the special case of Lemma \ref{divergence-free condition} (ii) for $%
n=1 $. See also Lemma 2.10 in \cite{Al1}.

\begin{corollary}
\label{en snabb rumsskala} For $n=1$ the space $D_{1}^{^{\ast }}$ in Lemma %
\ref{divergence-free condition} has the following property: any function $%
\theta \in D(\Omega )^{N}$ can also be expressed as the average of a
function $v$ in $D_{1}^{^{\ast }}$, such that%
\begin{equation*}
\theta =\int_{Y^{\ast }}v(x,y)dy\text{ and }\left\Vert v\right\Vert
_{L^{2}(\Omega \times Y^{\ast })^{N}}\leq C\left\Vert \theta \right\Vert
_{L^{2}(\Omega )^{N}}\text{.}
\end{equation*}
\end{corollary}

\begin{proof}
(Proof of Proposition \ref{P2,2}) We%
\'{}%
ll show that the limit $u\in L^{2}(0,T;H^{1}(\Omega ))$. We choose $v\in
D(\Omega \times (0,T)$; $C_{\sharp }^{\infty }(Y^{\ast }))^{N}$, such that $%
v\in D_{1}^{\ast }$ for\ any $t\in (0,T)$, in%
\begin{equation}
\int_{0}^{T}\int_{\Omega _{\varepsilon }}\nabla u_{\varepsilon }(x,t)\cdot
v(x,t,\frac{x}{\varepsilon })dxdt\text{.}  \label{epsint}
\end{equation}%
We assume that $\left\{ u_{\varepsilon }\right\} $ is bounded in $%
L^{2}(0,T;H^{1}\left( \Omega _{\varepsilon }\right) )$. Then $\left\{ \nabla
u_{\varepsilon }\right\} $ is bounded in $L^{2}(\Omega _{\varepsilon }\times
(0,T))^{N}$ and can be extended with zero to $\left\{ \widetilde{\nabla
u_{\varepsilon }}\right\} $ that is bounded in $L^{2}(\Omega \times
(0,T))^{N}$. We can now write (\ref{epsint}) as 
\begin{equation}
\int_{0}^{T}\int_{\Omega }\widetilde{\nabla u_{\varepsilon }}(x,t)\cdot
v(x,t,\frac{x}{\varepsilon })dxdt  \label{epsintomega}
\end{equation}%
and obtain by Theorem \ref{compactness result}, up to a subsequence,%
\begin{gather*}
\int_{0}^{T}\int_{\Omega }\widetilde{\nabla u_{\varepsilon }}(x,t)\chi
_{Y^{\ast }}(\frac{x}{\varepsilon })\cdot v(x,t,\frac{x}{\varepsilon })dxdt
\\
\rightarrow \int_{0}^{T}\int_{\Omega }\int_{0}^{1}\int_{Y^{\ast
}}w_{0}(x,t,y,s)v(x,t,y)dydsdxdt \\
\rightarrow \int_{0}^{T}\int_{\Omega }\int_{Y^{\ast }}\left(
\int_{0}^{1}w_{0}(x,t,y,s)ds\right) v(x,t,y)dydxdt \\
=\int_{0}^{T}\int_{\Omega }\int_{Y^{\ast }}\eta _{0}(x,t,y)\cdot
v(x,t,y)dydxdt
\end{gather*}%
for some $w_{0}\in L^{2}(\Omega \times (0,T)\times Y^{\ast }\times S)^{N}$
and $\eta _{0}=\int_{0}^{1}w_{0}ds\in L^{2}(\Omega \times (0,T)\times
Y^{\ast })^{N}$, when $\varepsilon $ goes to zero.

\noindent We now integrate by parts in (\ref{epsint}) and let $\varepsilon $
go to $0$. Further, we let $\widetilde{u_{\varepsilon }}$ be an extension of 
$u_{\varepsilon }$ from $\Omega _{\varepsilon }\times (0,T)$ to $\Omega
\times (0,T)$. Clearly, $\left\{ \widetilde{u_{\varepsilon }}\right\} $ is
bounded in $L^{2}(\Omega \times (0,T))$ if $\left\{ u_{\varepsilon }\right\} 
$ is bounded in $L^{2}(\Omega _{\varepsilon }\times (0,T))$. We get%
\begin{equation*}
\underset{\varepsilon \rightarrow 0}{\lim }\int_{0}^{T}\int_{\Omega
_{\varepsilon }}\nabla u_{\varepsilon }(x,t)\cdot v(x,t,\frac{x}{\varepsilon 
})dxdt=\underset{\varepsilon \rightarrow 0}{\lim }\int_{0}^{T}\int_{\Omega
_{\varepsilon }}-u_{\varepsilon }(x,t)\nabla \cdot v(x,t,\frac{x}{%
\varepsilon })dxdt
\end{equation*}%
\begin{equation*}
=\underset{\varepsilon \rightarrow 0}{\lim }\int_{0}^{T}\int_{\Omega }-%
\widetilde{u_{\varepsilon }}(x,t)\left( \nabla _{x}\cdot v(x,t,\frac{x}{%
\varepsilon })+\varepsilon ^{-1}\nabla _{y}\cdot v(x,t,\frac{x}{\varepsilon }%
)\right) dxdt
\end{equation*}%
\begin{equation*}
\underset{}{=\lim_{\varepsilon \rightarrow 0}}\int_{0}^{T}\int_{\Omega }-%
\widetilde{u_{\varepsilon }}(x,t)\nabla _{x}\cdot v(x,t,\frac{x}{\varepsilon 
})dxdt
\end{equation*}%
since $\nabla _{y}\cdot v=0.$ When $\varepsilon $ tends to zero we get from
assumption (\ref{assumption}) 
\begin{eqnarray}
&&-\int_{0}^{T}\int_{\Omega }\int_{Y}u(x,t)\chi _{Y^{\ast }}(y)\nabla
_{x}\cdot v(x,t,y)dydxdt  \label{Tvaskalegransdivtata} \\
&=&-\int_{0}^{T}\int_{\Omega }\int_{Y^{\ast }}u(x,t)\nabla _{x}\cdot
v(x,t,y)dydxdt\text{.}  \notag
\end{eqnarray}%
This means that 
\begin{gather}
\int_{0}^{T}\int_{\Omega }\int_{Y^{\ast }}\eta _{0}(x,t,y)\cdot
v(x,t,y)dydxdt  \label{Tvaskalegransparning} \\
=-\int_{0}^{T}\int_{\Omega }\int_{Y^{\ast }}u(x,t)\nabla _{x}\cdot
v(x,t,y)dydxdt  \notag \\
=\int_{0}^{T}\int_{\Omega }u(x,t)\left( \nabla _{x}\cdot \int_{Y^{\ast
}}v(x,t,y)dy\right) dxdt\text{.}  \notag
\end{gather}%
A simple modification of the proof of Lemma 4.13 (ii) in \cite{AlBr} (see
our Lemma \ref{divergence-free condition} and Corollary \ref{en snabb
rumsskala}) means that any $\theta \in D(\Omega \times (0,T))^{N}$ can be
expressed as 
\begin{equation*}
\theta (x,t)=\int_{Y^{\ast }}v(x,t,y)dy
\end{equation*}%
for some testfunction $v$ of the type we use. Moreover,%
\begin{equation*}
\left\Vert v\right\Vert _{L^{2}(\Omega \times (0,T)\times Y^{\ast
})^{N}}\leq C\left\Vert \theta \right\Vert _{L^{2}(\Omega \times (0,T))^{N}}%
\text{.}
\end{equation*}%
Hence, 
\begin{equation*}
\left\vert -\int_{0}^{T}\int_{\Omega }u(x,t)\left( \nabla _{x}\cdot \theta
(x,t)\right) dxdt\right\vert =\left\vert \int_{0}^{T}\int_{\Omega
}\int_{Y^{\ast }}\eta _{0}(x,t,y)\cdot v(x,t,y)dydxdt\right\vert
\end{equation*}%
\begin{equation*}
\leq C\left\Vert \eta _{0}\right\Vert _{L^{2}(\Omega \times T\times Y^{\ast
})^{N}}\left\Vert \theta \right\Vert _{L^{2}(\Omega \times (0,T))^{N}}\text{.%
}
\end{equation*}%
This means that%
\begin{equation*}
F(\theta )=-\int_{0}^{T}\int_{\Omega }u(x,t)\nabla _{x}\cdot \theta (x,t)dxdt
\end{equation*}%
is a bounded linear functional on $L^{2}(\Omega \times (0,T))^{N}$ for
arbitrary \newline
$\theta \in D(\Omega \times (0,T))^{N}$. By continuous extension, (e.g.
Theorem 6.14 in \cite{Gr}) this also applies to all $\theta \in L^{2}(\Omega
\times 0,T)^{N}$. This means that, for some $r_{0}\in L^{2}(\Omega \times
(0,T))^{N}$, 
\begin{equation*}
F(\theta )=-\int_{0}^{T}\int_{\Omega }u(x,t)\left( \nabla _{x}\cdot \theta
(x,t)\right) dxdt=\int_{0}^{T}\int_{\Omega }r_{0}(x,t)\cdot \theta (x,t)dxdt
\end{equation*}%
and hence the distributional gradient $\nabla u$ of $u$ belongs to $%
L^{2}(\Omega \times (0,T))^{N}$. We already knew that $u\in L^{2}(\Omega
\times (0,T))$ and hence $u\in L^{2}(0,T;H^{1}\left( \Omega \right) )$.
\end{proof}

We also cite the Lemma 4.14 of Allaire and Briane in \cite{AlBr}, because we
will use it in the following theorem.

\begin{lemma}
\label{cited} Let $H^{\ast }$ be the space of 
\'{}%
generalised%
\'{}
divergence-free functions in $L^{2}[\Omega ;$ $L_{\sharp }^{2}(Y_{1}^{\ast
}\times \ldots \times Y_{n}^{\ast })]^{N}$ defined by 
\begin{equation*}
v\in H^{\ast }\Leftrightarrow \left\{ 
\begin{array}{l}
\func{div}_{y_{n}}v=0\text{ in }Y_{n}^{\ast } \\ 
v\cdot n=0\text{ on }\partial Y_{n}^{\ast }-\partial Y_{n}%
\end{array}%
\right.
\end{equation*}%
and 
\begin{equation*}
\left\{ 
\begin{array}{l}
\int_{Y_{k+1}}\ldots \int_{Y_{n}}\func{div}_{y_{k}}v=0\text{ in }Y_{k}^{\ast
} \\ 
\int_{Y_{k+1}}\ldots \int_{Y_{n}}v\cdot n=0\text{ on }\partial Y_{k}^{\ast
}-\partial Y_{k}%
\end{array}%
\right.
\end{equation*}%
for all $1\leqq k\leqq n-1$.

The subspace $H^{\ast }$ has the following properties:\smallskip

(i) $(D^{\ast })^{N}\cap H^{\ast }$ is dense into $H^{\ast }$.

(ii)\ The orthogonal of $H^{\ast }$ is 
\begin{equation*}
(H^{\ast })^{\perp }=\left\{ \sum_{k=1}^{n}\nabla
_{y_{k}}q_{k}(x,y_{1},\ldots ,y_{k})\text{ with }q_{k}\in L^{2}\left[ \Omega
\times Y_{1}^{\ast }\times \ldots \times Y_{k-1}^{\ast };H_{\sharp
}^{1}(Y_{k}^{\ast })/%
\mathbb{R}
\right] \right\} \text{.}
\end{equation*}%
\ \ \ \ \ \ \ \ \ \ \ \ \ \ \ \ \ \ \ \ \ \ \ \ \ \ \ \ \ \ \ \ \ \ \ \ \ \
\ \ \ \ \ \ \ \ \ \ \ \ \ \ \ \ \ \ \ \ \ \ \ \ \ \ \ \ \ \ \ \ \ \ \ \ \ \
\ \ \ \ \ \ \ \ \ \ \ \ \ \ \ \ \ \ \ \ \ \ \ \ \ \ \ \ \ \ \ \ \ \ \ \ \ \
\ \ \ \ \ \ \ \ \ \ \ \ \ \ \ \ \ \ \ \ \ \ \ \ \ \ \ \ \ \ \ \ \ \ \ \ \ \
\ \ \ \ \ \ \ \ \ \ \ \ \ \ \ \ \ \ \ \ \ \ \ \ \ \ \ \ \ \ \ \ \ \ \ \ \ \
\ \ \ \ \ \ \ \ \ \ \ \ \ \ \ \ \ \ \ \ \ \ \ \ \ \ \ \ \ \ \ \ \ \ \ \ \ \
\ \ \ \ \ \ \ \ \ \ \ \ \ \ \ \ \ \ \ \ \ \ \ \ \ \ \ \ \ \ \ \ \ \ \ \ \ \
\ \ \ \ \ \ \ \ \ \ \ \ \ \ \ \ \ \ \ \ \ \ \ \ \ \ \ \ \ \ \ \ \ \ \ \ \ \
\ \ \ \ \ \ \ \ \ \ \ \ \ \ \ \ \ \ \ \ \ \ \ \ \ \ \ \ \ \ \ \ \ \ \ \ \ \
\ \ \ \ \ \ \ \ \ \ \ \ \ \ \ \ \ \ \ \ \ \ \ \ \ \ \ \ \ \ \ \ \ \ \ \ \ \
\ \ \ \ \ \ \ \ \ \ \ \ \ \ \ \ \ \ \ \ \ \ \ \ \ \ \ \ \ \ \ \ \ \ \ \ \ \
\ \ \ \ \ \ \ \ \ \ \ \ \ \ \ \ \ \ \ \ \ \ \ \ \ \ \ \ \ \ \ \ \ \ \ \ \ \
\ \ \ \ \ \ \ \ \ \ \ \ \ \ \ \ \ \ \ \ \ \ \ \ \ \ \ \ \ \ \ \ \ \ \ \ \ \
\ \ \ \ \ \ \ \ \ \ \ \ \ \ \ \ \ \ \ \ \ \ \ \ \ \ \ \ \ \ \ \ \ \ \ \ \ \
\ \ \ \ \ \ \ \ \ \ \ \ \ \ \ \ \ \ \ \ \ \ \ \ \ \ \ \ \ \ \ \ \ \ \ \ \ \
\ \ \ \ \ \ \ \ \ \ \ \ \ \ \ \ \ \ \ \ \ \ \ \ \ \ \ \ \ \ \ \ \ \ 
\end{lemma}

We first find a characterization of the $(2,2)$-scale limit for $\left\{
\nabla u_{\varepsilon }\right\} $ under certain assumptions.

\begin{theorem}
\label{first in line}Assume that $\left\{ u_{\varepsilon }\right\} $ is
bounded in $L^{2}(0,T;H^{1}(\Omega _{\varepsilon }))$ and for any $v_{1}\in
D(\Omega )$, $c_{1}\in D(0,T)$, $c_{2}\in C_{\sharp }^{\infty }(0,1)$ 
\begin{equation}
\varepsilon ^{r}\int_{0}^{T}\int_{\Omega _{\varepsilon }}u_{\varepsilon
}(x,t)v_{1}(x)\partial _{t}\left( c_{1}(t)c_{2}\left( \frac{t}{\varepsilon
^{r}}\right) \right) dxdt\rightarrow 0  \label{name 3}
\end{equation}%
for some $r>0.$

\noindent Then, up to a subsequence,%
\begin{eqnarray}
&&\int_{0}^{T}\int_{\Omega _{\varepsilon }}\nabla u_{\varepsilon }(x,t)\cdot
v(x,t,\frac{x}{\varepsilon },\frac{t}{\varepsilon ^{r}})dxdt  \label{name22}
\\
&\rightarrow &\int_{0}^{T}\int_{\Omega }\int_{0}^{1}\int_{Y^{\ast }}(\nabla
u(x,t)+\nabla _{y}u_{1}(x,t,y,s))\cdot v(x,t,y,s)dydsdxdt,  \notag
\end{eqnarray}%
\textit{where }$u\in L^{2}(0,T;H^{1}(\Omega ))$, $u_{1}\in L^{2}[\Omega
\times (0,T)\times (0,1);$ $H_{\sharp }^{1}(Y^{\ast })/%
\mathbb{R}
]$, for any $v\in L^{2}(\Omega \times (0,T);C_{\sharp }\left( \mathcal{Y}%
_{1,1}\right) )^{N}$.
\end{theorem}

\begin{proof}
We first show that the $(2,2)$-scale limit $u_{0}$ for $\overset{}{%
\widetilde{u_{\varepsilon }}}$ does not depend on $y$. Indeed, integration
by parts gives us%
\begin{eqnarray}
&&\varepsilon \int_{0}^{T}\int_{\Omega _{\varepsilon }}\nabla u_{\varepsilon
}(x,t)\cdot v_{1}(x)v_{2}\left( \frac{x}{\varepsilon }\right)
c_{1}(t)c_{2}\left( \frac{t}{\varepsilon ^{r}}\right) dxdt  \notag \\
&=&-\int_{0}^{T}\int_{\Omega }\overset{}{\widetilde{u_{\varepsilon }}}(x,t)%
\left[ v_{1}(x)\nabla _{y}\cdot v_{2}\left( \frac{x}{\varepsilon }\right)
+\varepsilon \nabla _{x}v_{1}(x)\cdot v_{2}\left( \frac{x}{\varepsilon }%
\right) \right]  \label{name2} \\
&&\times \left( c_{1}(t)c_{2}\left( \frac{t}{\varepsilon ^{r}}\right)
\right) dxdt\text{,}  \notag
\end{eqnarray}%
for any $v_{1}\in D(\Omega )$ and $v_{2}\in C_{\sharp }^{\infty }(Y)^{N}$
with $v_{2}(y)=0$ for $y\in Y-Y^{\ast }$ and $c_{1}\in D(0,T),$ $c_{2}\in
C_{\sharp }^{\infty }(0,1)$. Passing to the limit on both sides leads to%
\begin{equation*}
-\int_{0}^{T}\int_{\Omega }\int_{0}^{1}\int_{Y^{\ast
}}u_{0}(x,t,y,s)v_{1}(x)\nabla _{y}\cdot v_{2}(y)c_{1}(t)c_{2}(s)dydsdxdt=0.
\end{equation*}%
This implies that $u_{0}$ does not depend on $y$ in $Y^{\ast }$, i.e. there\
exists \newline
$u\ \in L^{2}(\Omega \times (0,T)\times (0,1))$ such that%
\begin{equation*}
u_{0}(x,t,y,s)=u(x,t,s)\chi _{Y^{\ast }}(y)\text{.}
\end{equation*}%
Let us then show that, by assumption%
\begin{equation*}
\varepsilon ^{r}\int_{0}^{T}\int_{\Omega _{\varepsilon }}u_{\varepsilon
}(x,t)v_{1}(x)\partial _{t}\left( c_{1}(t)c_{2}\left( \frac{t}{\varepsilon
^{r}}\right) \right) dxdt\rightarrow 0\text{,}
\end{equation*}%
$u_{0}$ does not depend of $s$. We rewrite the form (\ref{name 3}) as%
\begin{eqnarray*}
&&\int_{0}^{T}\int_{\Omega }\varepsilon ^{r}\widetilde{u}_{\varepsilon
}(x,t)v_{1}(x)\left( \partial _{t}c_{1}(t)\right) c_{2}(\frac{t}{\varepsilon
^{r}})dxdt \\
&&+\int_{0}^{T}\int_{\Omega }\widetilde{u}_{\varepsilon
}(x,t)v_{1}(x)c_{1}(t)\left( \partial _{s}c_{2}^{{}}\left( \frac{t}{%
\varepsilon ^{r}}\right) \right) dxdt\text{,}
\end{eqnarray*}%
where $v_{1}\in D(\Omega )$, $c_{1}\in D(0,T)$ and $c_{2}\in C_{\sharp
}^{\infty }(0,1)$.

As $\varepsilon $ tends to zero, we obtain that%
\begin{equation*}
\int_{0}^{T}\int_{\Omega }\int_{0}^{1}\int_{Y^{^{\ast
}}}u(x,t,s)v_{1}(x)c_{1}(t)\partial _{s}c_{2}(s)dydsdxdt=0\text{.}
\end{equation*}%
Finally, applying the variational lemma, we get%
\begin{equation*}
\int_{0}^{1}u(x,t,s)\partial _{s}c_{2}(s)ds=0\text{,}
\end{equation*}%
for a.e. $(x,t)\in \Omega \times (0,T)$. We deduce now that $u_{0}$ does not
depend on the local time variable $s$. Hence, 
\begin{equation*}
\widetilde{u}_{\varepsilon }(x,t)\overset{2,2}{\rightharpoonup }u(x,t)\chi
_{Y^{\ast }}(y)
\end{equation*}%
and Proposition \ref{P2,2} yields that $u\in L^{2}(0,T;H^{1}(\Omega ))$. We
now let $v\in (D^{\ast })^{N}\cap H^{\ast }$ for $n=1$ (see Lemma \ref{cited}%
) and again $c_{1}\in D(0,T),$ $c_{2}\in C_{\sharp }^{\infty }(0,1).$
Integrating by parts in $\Omega _{\varepsilon }$ gives%
\begin{gather*}
\int_{0}^{T}\int_{\Omega _{\varepsilon }}\nabla u_{\varepsilon }(x,t)\cdot
v\left( x,\frac{x}{\varepsilon }\right) c_{1}(t)c_{2}\left( \frac{t}{%
\varepsilon ^{r}}\right) dxdt \\
=-\int_{0}^{T}\int_{\Omega _{\varepsilon }}u_{\varepsilon }(x,t)\nabla
_{x}\cdot v\left( x,\frac{x}{\varepsilon }\right) c_{1}(t)c_{2}\left( \frac{t%
}{\varepsilon ^{r}}\right) dxdt.
\end{gather*}%
Hence, passing to the two-scale limit yields that, for some $w_{0}\in
L^{2}(\Omega \times (0,T)\times Y^{\ast }\times (0,1))^{N}$,%
\begin{gather}
\int_{0}^{T}\int_{\Omega }\int_{0}^{1}\int_{Y^{\ast }}w_{0}(x,t,y,s)\cdot
v(x,y)c_{1}(t)c_{2}(s)dydxdsdt  \label{name 4} \\
=-\int_{0}^{T}\int_{\Omega }\int_{0}^{1}\int_{Y^{\ast }}u(x,t)\nabla
_{x}\cdot v(x,y)c_{1}(t)c_{2}(s)dydxdsdt  \notag \\
=\int_{0}^{T}\int_{\Omega }\int_{0}^{1}\int_{Y^{\ast }}\nabla u(x,t)\cdot
v(x,y)c_{1}(t)c_{2}(s)dydxdsdt.  \notag
\end{gather}%
We have 
\begin{equation}
\int_{0}^{T}\int_{\Omega }\int_{0}^{1}\int_{Y^{\ast }}(w_{0}(x,t,y,s)-\nabla
u(x,t))\cdot v(x,y)c_{1}(t)c_{2}(s)dydsdxdt=0\text{,}  \label{name 6}
\end{equation}%
which means that a.e. on $(0,T)\times (0,1)$%
\begin{equation*}
\int_{\Omega }\int_{Y^{\ast }}(w_{0}(x,t,y,s)-\nabla u(x,t))\cdot
v(x,y)dydx=0\text{.}
\end{equation*}%
Moreover, the orthogonal of $H^{\ast }$ are gradients. See Lemma \ref{cited}
for n=1. See also the proof of Theorem 2.9 in \cite{Al1}. This implies that
there exists a unique function $u_{1}$ in $L^{2}\left[ \Omega \times
(0,T)\times (0,1);\text{ }H_{\sharp }^{1}(Y^{\ast })/%
\mathbb{R}
\right] $ such that%
\begin{equation*}
\widetilde{\nabla u}_{\varepsilon }(x,t)\overset{2,2}{\rightharpoonup }%
(\nabla u(x,t)+\nabla _{y}u_{1}(x,t,y,s))\chi _{Y^{\ast }}(y).
\end{equation*}%
See also \cite{Ho}.
\end{proof}

From the Theorem \ref{first in line} above, as a consequence, we have the
following corollary.

\begin{corollary}
\label{Corr1} Assume that $\left\{ u_{\varepsilon }\right\} $ is bounded in $%
L^{2}(0,T;H^{1}(\Omega _{\varepsilon }))$ and that (\ref{name 3}) holds in
Theorem \ref{first in line}. T\textit{hen}%
\begin{gather}
\int_{0}^{T}\int_{\Omega _{\varepsilon }}\varepsilon ^{-1}u_{\varepsilon
}(x,t)v_{1}(x)v_{2}\left( \frac{x}{\varepsilon }\right) c_{1}(t)c_{2}\left( 
\frac{t}{\varepsilon ^{r}}\right) dxdt  \label{equation 1} \\
\rightarrow \int_{0}^{T}\int_{\Omega }\int_{0}^{1}\int_{Y^{\ast
}}u_{1}(x,t,y,s)v_{1}(x)v_{2}(y)c_{1}(t)c_{2}(s)dydsdxdt  \notag
\end{gather}%
for $v_{1}\in D(\Omega )$, $v_{2}\in C_{\sharp }^{\infty }(Y^{\ast })/%
\mathbb{R}
$, $c_{1}\in D(0,T)$ and $c_{2}\in C_{\sharp }^{\infty }(0,1)$.
\end{corollary}

\begin{proof}
First we note that any $v_{2}\in C_{\sharp }^{\infty }(Y^{\ast })/%
\mathbb{R}
$ can be expressed as%
\begin{equation}
v_{2}^{{}}(y)=\bigtriangleup _{y}\rho (y)=\nabla _{y}\cdot \left( \nabla
_{y}\rho (y)\right)   \label{equation 2}
\end{equation}%
for some $\rho \in C_{\sharp }^{\infty }(Y^{\ast })/%
\mathbb{R}
$. Furthermore, we note that we can find $\rho \in C_{\sharp }^{\infty
}(Y^{\ast })/%
\mathbb{R}
$ through%
\begin{eqnarray*}
\bigtriangleup \rho (y) &=&v_{2}^{{}}(y)\text{, }y\in Y^{^{\ast }}\text{,} \\
\nabla _{y}\rho (y)\cdot n &=&0\text{, }y\in \partial Y^{^{\ast }}-\partial Y%
\text{.}
\end{eqnarray*}%
By (\ref{equation 2}) the left-hand side of (\ref{equation 1}) can be
expressed as%
\begin{eqnarray*}
&&\int_{0}^{T}\int_{\Omega _{\varepsilon }}\varepsilon ^{-1}u_{\varepsilon
}(x,t)v_{1}(x)c_{1}(t)c_{2}\left( \frac{t}{\varepsilon ^{r}}\right) \left(
\nabla _{y}\cdot \nabla _{y}\rho \right) \left( \frac{x}{\varepsilon }%
\right) dxdt \\
&=&\int_{0}^{T}\int_{\Omega _{\varepsilon }}u_{\varepsilon
}(x,t)v_{1}(x)c_{1}(t)c_{2}\left( \frac{t}{\varepsilon ^{r}}\right) \nabla
\cdot \left( \nabla _{y}\rho \left( \frac{x}{\varepsilon }\right) \right)
dxdt\text{.}
\end{eqnarray*}%
Integrating by parts with respect to $x$ we obtain%
\begin{eqnarray}
&&-\int_{0}^{T}\int_{\Omega _{\varepsilon }}\nabla u_{\varepsilon
}(x,t)v_{1}(x)c_{1}(t)c_{2}\left( \frac{t}{\varepsilon ^{r}}\right) \cdot
\nabla _{y}\rho \left( \frac{x}{\varepsilon }\right)   \label{equation 4} \\
&&+u_{\varepsilon }(x,t)\nabla v_{1}(x)c_{1}(t)c_{2}\left( \frac{t}{%
\varepsilon ^{r}}\right) \cdot \nabla _{y}\rho \left( \frac{x}{\varepsilon }%
\right) dxdt\text{.}  \notag
\end{eqnarray}%
Passing to the limit in the first term we get, up to a subsequence, that%
\begin{gather*}
-\int_{0}^{T}\int_{\Omega }\int_{0}^{1}\int_{Y^{\ast }}(\nabla u(x,t)+\nabla
_{y}u_{1}(x,t,y,s))v_{1}(x)c_{1}(t)c_{2}(s)\cdot \nabla _{y}\rho (y) \\
+u(x,t)\nabla v_{1}(x)c_{1}(t)c_{2}(s)\cdot \nabla _{y}\rho \left( y\right)
dydsdxdt\text{.}
\end{gather*}%
Integration by parts in the last term with respect to $x$, gives%
\begin{equation}
-\int_{0}^{T}\int_{\Omega }\int_{0}^{1}\int_{Y^{\ast }}\nabla
_{y}u_{1}(x,t,y,s)v_{1}(x)c_{1}(t)c_{2}(s)\cdot \nabla _{y}\rho \left(
y\right) dydsdxdt\text{.}  \notag
\end{equation}%
Finally, integrating by parts with respect to $y$ we obtain%
\begin{gather*}
\int_{0}^{T}\int_{\Omega }\int_{0}^{1}\int_{Y^{\ast
}}u_{1}(x,t,y,s)v_{1}(x)c_{1}(t)c_{2}(s)\nabla _{y}\cdot (\nabla _{y}\rho
(y))dydsdxdt \\
=\int_{0}^{T}\int_{\Omega }\int_{0}^{1}\int_{Y^{\ast
}}u_{1}(x,t,y,s)v_{1}(x)c_{1}(t)c_{2}(s)v_{2}(y)dydsdxdt\text{,}
\end{gather*}%
which is the right-hand side of (\ref{equation 1}).
\end{proof}

\section{Homogenization result}

\noindent We will investigate the parabolic problem with spatial and
temporal oscillations%
\begin{eqnarray}
\partial _{t}u_{\varepsilon }(x,t)-\nabla \cdot \left( A\left( \frac{x}{%
\varepsilon },\frac{t}{\varepsilon ^{2}}\right) \nabla u_{\varepsilon
}(x,t)\right) &=&f_{\varepsilon }(x,t)\text{ in }\Omega _{\varepsilon
}\times (0,T)\text{,}  \notag \\
u_{\varepsilon }(x,t) &=&0\text{ on }\partial \Omega \times (0,T)\text{,}
\label{see1} \\
A\left( \frac{x}{\varepsilon },\frac{t}{\varepsilon ^{2}}\right) \nabla
u_{\varepsilon }(x,t)\cdot n &=&0\text{ on }(\partial \Omega _{\varepsilon
}-\partial \Omega )\times (0,T)\text{,}  \notag \\
u_{\varepsilon }(x,0) &=&u_{\varepsilon }^{0}(x)\text{ in }\Omega
_{\varepsilon }\text{,}  \notag
\end{eqnarray}%
where$\ $%
\begin{equation}
\tilde{u}_{\varepsilon }^{0}\rightharpoonup u^{0}\text{ in }L^{2}(\Omega )
\label{view2}
\end{equation}%
and 
\begin{equation}
\tilde{f}_{\varepsilon }\rightharpoonup f\text{ in }L^{2}(\Omega \times
(0,T))\text{.}  \label{view3}
\end{equation}%
\ Moreover, we assume that

\noindent (H1) $A\in C_{\sharp }\left( \mathcal{Y}_{1,1}\right) ^{N\times N}$

\noindent (H2) $A(y,s)\xi \cdot \xi \geq \alpha \left\vert \xi \right\vert
^{2}$

\noindent for all $\left( y,s\right) \in 
\mathbb{R}
^{N}\times 
\mathbb{R}
,$ all $\xi \in 
\mathbb{R}
^{N}$ and some $\alpha >0$.

\noindent We introduce the space 
\begin{equation*}
V_{\varepsilon }=\left\{ v\in H^{1}(\Omega _{\varepsilon })\left\vert v(x)=0%
\text{ on }\partial \Omega \right. \right\}
\end{equation*}%
\noindent with the $H^{1}(\Omega _{\varepsilon })$-norm.

\noindent Under these conditions the problem (\ref{see1}) allows a unique
solution $\left\{ u_{\varepsilon }\right\} $ bounded in $L^{\infty
}(0,T;L^{2}(\Omega _{\varepsilon }))$ and $L^{2}(0,T;V_{\varepsilon })$, see
Section 3 in \cite{DN}.

\noindent We are now prepared to prove the following theorem.

\begin{theorem}
\label{svarast}Let $\left\{ u_{\varepsilon }\right\} $ be a sequence\ of
solutions in $L^{2}(0,T;V_{\varepsilon })$ to (\ref{see1}). Then it holds
that 
\begin{equation*}
\widetilde{u_{\varepsilon }}(x,t)\overset{2,2}{\rightharpoonup }u(x,t)\chi
_{Y^{\ast }}\left( y\right)
\end{equation*}%
and 
\begin{equation*}
\widetilde{\nabla u_{\varepsilon }}(x,t)\overset{2,2}{\rightharpoonup }%
(\nabla u(x,t)+\nabla _{y_{1}}u_{1}(x,t,y,s))\chi _{Y^{\ast }}\left(
y\right) \text{,}
\end{equation*}%
where $(u,u_{1})\in L^{2}(0,T;H_{0}^{1}(\Omega ))\times L^{2}\left[ \Omega
\times (0,T)\times (0,1);\text{ }H_{\sharp }^{1}(Y^{\ast })/%
\mathbb{R}
)\right] $\newline
is the unique solution of the following two-scale homogenized system:%
\begin{eqnarray}
\mu (Y^{\ast })\partial _{t}u(x,t)-\nabla \cdot (b\nabla u(x,t)) &=&f(x,t)%
\text{ in }\Omega \times (0,T)\text{,}  \notag \\
\text{\ \ \ }u(x,t) &=&0\text{ on }\partial \Omega \times (0,T)\text{,}
\label{view11} \\
u(x,0) &=&(\mu (Y^{\ast }))^{-1}u^{0}(x)\text{ in }\Omega \text{,}  \notag
\end{eqnarray}%
where%
\begin{equation*}
b\nabla u=\int_{0}^{1}\int_{Y^{\ast }}A(y,s)(\nabla u(x,t)+\nabla
_{y}u_{1}(x,t,y,s))dyds
\end{equation*}%
and $u_{1}$ solves our local problem%
\begin{equation}
\partial _{s}u_{1}(x,t,y,s)-\nabla _{y}\cdot (A(y,s)(\nabla u(x,t)+\nabla
_{y}u_{1}(x,t,y,s)))=0\text{ in }Y^{^{\ast }}\times (0,1)
\label{locproblem1}
\end{equation}%
\begin{equation}
(A(y,s)\left[ \nabla u(x,t)+\nabla _{y}u_{1}(x,t,y,s)\right] )\cdot n=0\text{
on }(\partial Y^{\ast }-\partial Y)\times (0,1)  \label{posrlocal}
\end{equation}%
for a.e. $(x,t)\in \Omega \times (0,T).$
\end{theorem}

\begin{remark}
\label{sepvar} Let us point out, that by the uniqueness of solutions to the
above system, the entire sequences $\left\{ u_{\varepsilon }\right\} ,$ $%
\left\{ \nabla u_{\varepsilon }\right\} $ two-scale converge (see \cite{Nan}%
). Two-scale homogenized system (\ref{view11}) can be decoupled to a
familiar type of homogenized system. Using the ansatz%
\begin{equation*}
u_{1}(x,t,y,s)=\nabla u(x,t)\cdot z(y,s)\text{,}
\end{equation*}%
the variable separated version of the local problem becomes%
\begin{equation*}
\partial _{s}z_{j}(y,s)-\nabla _{y}\cdot \left( A(y,s)(e_{j}+\nabla
_{y}z_{j}(y,s)\right) )=0\text{ in }Y^{\ast }\times (0,1)\text{,}
\end{equation*}%
where $j=1,\ldots ,N.$ Analogously for (\ref{posrlocal}), we obtain%
\begin{equation*}
A(y,s)\left( e_{j}+\nabla _{y}z_{j}(y,s)\right) \cdot n=0\text{ on }%
(\partial Y^{\ast }-\partial Y)\times (0,1)\text{.}
\end{equation*}%
We get the expression for the homogenized coefficients
\end{remark}

\begin{equation*}
b_{ij}=\int_{0}^{1}\int_{Y^{\ast
}}A_{ij}(y,s)+\sum_{k=1}^{N}A_{ik}(y,s)\partial _{y_{k}}z_{j}(y,s)dyds\text{.%
}
\end{equation*}

\begin{proof}
We carry out a homogenization procedure for (\ref{see1}). The corresponding
weak form states that we are searching for a unique $u_{\varepsilon }$ in $%
L^{2}(0,T;V_{\varepsilon })$ such that%
\begin{gather}
\int_{0}^{T}\int_{\Omega _{\varepsilon }}-u_{\varepsilon }(x,t)v(x)\partial
_{t}c(t)+A\left( \frac{x}{\varepsilon },\frac{t}{\varepsilon ^{2}}\right)
\nabla u_{\varepsilon }(x,t)\cdot \nabla v(x)c(t)dxdt  \label{weakform} \\
=\int_{0}^{T}\int_{\Omega _{\varepsilon }}f_{\varepsilon }(x,t)v(x)c(t)dxdt,%
\text{ }v\in V_{\varepsilon },\text{ }c\in D(0,T)\text{.}  \notag
\end{gather}%
We want to prove the weak form of the homogenized problem (\ref{view11}). To
see that (\ref{name 3}) is satisfied for $r>0$ and hence for $r=2$, we
conclude that for the choice of test functions in (\ref{name 3}): 
\begin{gather*}
-\varepsilon ^{r}\int_{0}^{T}\int_{\Omega _{\varepsilon }}u_{\varepsilon
}(x,t)v(x)\partial _{t}\left( c_{1}(t)c_{2}\left( \frac{t}{\varepsilon ^{r}}%
\right) \right) dxdt \\
=-\varepsilon ^{r}\int_{0}^{T}\int_{\Omega _{\varepsilon }}A\left( \frac{x}{%
\varepsilon },\frac{t}{\varepsilon ^{r}}\right) \nabla u_{\varepsilon
}(x,t)\cdot \nabla v(x) \\
\times c_{1}(t)c_{2}\left( \frac{t}{\varepsilon ^{r}}\right) +f_{\varepsilon
}(x,t)v(x)c_{1}(t)c_{2}\left( \frac{t}{\varepsilon ^{r}}\right)
dxdt\rightarrow 0\text{.}
\end{gather*}%
We choose $v(x,t)=v_{1}(x)c_{1}(t)$, $v_{1}\in D(\Omega ),$ $c_{1}\in D(0,T)$
in (\ref{weakform}). When passing to the limit and letting $\varepsilon
\rightarrow 0,$ we obtain through Theorem \ref{first in line} and (\ref%
{view3})%
\begin{gather}
\int_{0}^{T}\int_{\Omega }\int_{0}^{1}\int_{Y^{\ast
}}-u(x,t)v_{1}(x)\partial _{t}c_{1}(t)  \notag \\
+A(y,s)(\nabla u(x,t)+\nabla _{y}u_{_{1}}(x,t,y,s))\cdot \nabla
v_{1}(x)c_{1}(t)dydsdxdt  \label{equat1} \\
=\int_{0}^{T}\int_{\Omega }f(x,t)v_{1}(x)c_{1}(t)dxdt\text{.}  \notag
\end{gather}%
To find the local problem we choose the test functions in (\ref{weakform}) as%
\begin{eqnarray*}
v(x) &=&\varepsilon v_{1}(x)v_{2}\left( \frac{x}{\varepsilon }\right) \\
c(t) &=&c_{1}(t)c_{2}\left( \frac{t}{\varepsilon ^{2}}\right)
\end{eqnarray*}%
where $v_{1}\in $ $D(\Omega )$, $v_{2}\in C_{\sharp }^{\infty }(Y^{\ast })/%
\mathbb{R}
,$ $c_{1}\in D(0,T)$ and $c_{2}\in C_{\sharp }^{\infty }(0,1)$. This gives
us 
\begin{gather}
\int_{0}^{T}\int_{\Omega _{\varepsilon }}-u_{\varepsilon }(x,t)\varepsilon
v_{1}(x)v_{2}\left( \frac{x}{\varepsilon }\right) \partial _{t}\left(
c_{1}(t)c_{2}\left( \frac{t}{\varepsilon ^{2}}\right) \right)  \notag \\
+A\left( \frac{x}{\varepsilon },\frac{t}{\varepsilon ^{2}}\right) \nabla
u_{\varepsilon }(x,t)\cdot \nabla \left( \varepsilon v_{1}(x)v_{2}\left( 
\frac{x}{\varepsilon }\right) \right) c_{1}(t)c_{2}\left( \frac{t}{%
\varepsilon ^{2}}\right) dxdt  \label{equat2} \\
=\int_{0}^{T}\int_{\Omega _{\varepsilon }}f_{\varepsilon }(x,t)\varepsilon
v_{1}(x)v_{2}\left( \frac{x}{\varepsilon }\right) c_{1}(t)c_{2}\left( \frac{t%
}{\varepsilon ^{2}}\right) dxdt\text{.}  \notag
\end{gather}%
Carring out the differentiations yields%
\begin{gather*}
\int_{0}^{T}\int_{\Omega _{\varepsilon }}-u_{\varepsilon
}(x,t)v_{1}(x)v_{2}\left( \frac{x}{\varepsilon }\right) \left( \varepsilon
\partial _{t}c_{1}(t)c_{2}\left( \frac{t}{\varepsilon ^{2}}\right)
+\varepsilon ^{-1}c_{1}(t)\partial _{s}c_{2}\left( \frac{t}{\varepsilon ^{2}}%
\right) \right) \\
+A\left( \frac{x}{\varepsilon },\frac{t}{\varepsilon ^{2}}\right) \nabla
u_{\varepsilon }(x,t)\cdot \left( \varepsilon \nabla v_{1}(x)v_{2}\left( 
\frac{x}{\varepsilon }\right) +v_{1}(x)\nabla _{y_{1}}v_{2}\left( \frac{x}{%
\varepsilon }\right) \right) c_{1}(t)c_{2}\left( \frac{t}{\varepsilon ^{2}}%
\right) dxdt \\
=\int_{0}^{T}\int_{\Omega _{\varepsilon }}f_{\varepsilon }(x,t)\varepsilon
v_{1}(x)v_{2}\left( \frac{x}{\varepsilon }\right) c_{1}(t)c_{2}\left( \frac{t%
}{\varepsilon ^{2}}\right) dxdt\text{.}
\end{gather*}%
If we keep only terms that do not tend to zero there remains%
\begin{gather}
\underset{\varepsilon \rightarrow 0}{\lim }\int_{0}^{T}\int_{\Omega
_{\varepsilon }}-\varepsilon ^{-1}u_{\varepsilon }(x,t)v_{1}(x)v_{2}\left( 
\frac{x}{\varepsilon }\right) c_{1}(t)\partial _{s}c_{2}\left( \frac{t}{%
\varepsilon ^{2}}\right)  \label{property1} \\
+A\left( \frac{x}{\varepsilon },\frac{t}{\varepsilon ^{2}}\right) \nabla
u_{\varepsilon }(x,t)\cdot v_{1}(x)\nabla _{y}v_{2}\left( \frac{x}{%
\varepsilon }\right) c_{1}(t)c_{2}\left( \frac{t}{\varepsilon ^{2}}\right)
dxdt=0\text{.}  \notag
\end{gather}%
So far, we know by the Corollary \ref{Corr1} that we can get%
\begin{gather}
\int_{0}^{T}\int_{\Omega _{\varepsilon }}-\varepsilon ^{-1}u_{\varepsilon
}(x,t)v_{1}(x)v_{2}\left( \frac{x}{\varepsilon }\right) c_{1}(t)\partial
_{s}c_{2}\left( \frac{t}{\varepsilon ^{2}}\right) dxdt  \notag \\
\rightarrow \int_{0}^{T}\int_{\Omega }\int_{0}^{1}\int_{Y^{\ast
}}-u_{1}(x,t,y,s)v_{1}(x)v_{2}(y)c_{1}(t)\partial _{s}c_{2}(s)dydsdxdt
\label{property2}
\end{gather}%
and by Theorem \ref{first in line}%
\begin{gather}
\int_{0}^{T}\int_{\Omega _{\varepsilon }}A\left( \frac{x}{\varepsilon },%
\frac{t}{\varepsilon ^{2}}\right) \nabla u_{\varepsilon }(x,t)\cdot
v_{1}(x)\nabla _{y}v_{2}\left( \frac{x}{\varepsilon }\right)
c_{1}(t)c_{2}\left( \frac{t}{\varepsilon ^{2}}\right) dxdt  \notag \\
\rightarrow \int_{0}^{T}\int_{\Omega }\int_{0}^{1}\int_{Y^{\ast
}}A(y,s)(\nabla u(x,t)+\nabla _{y}u_{1}(x,t,y,s))  \label{property3} \\
\cdot v_{1}(x)\nabla _{y}v_{2}\left( y\right) c_{1}(t)c_{2}\left( s\right)
dydsdxdt\text{.}  \notag
\end{gather}%
Hence, putting (\ref{property1}), (\ref{property2}) and (\ref{property3})
together we arrive at the local problem%
\begin{gather*}
\int_{0}^{T}\int_{\Omega }\int_{0}^{1}\int_{Y^{\ast
}}-u_{1}(x,t,y,s)v_{1}(x)v_{2}(y)c_{1}(t)\partial _{s}c_{2}(s) \\
+A(y,s)(\nabla u(x,t)+\nabla _{y}u_{1}(x,t,y,s)) \\
\cdot v_{1}(x)\nabla _{y}v_{2}\left( y\right) c_{1}(t)c_{2}\left( s\right)
dydsdxdt=0\text{,}
\end{gather*}%
which is the weak form of (\ref{locproblem1}). The proof is complete.
\end{proof}

\begin{remark}
The corresponding limit for the initial condition (\ref{view2}) is found in
the same way as in the proof of Theorem 4.1 in \cite{DN}.
\end{remark}

\begin{acknowledgement}
The author thanks Prof. Anders Holmbom, Mid Sweden University, \"{O}%
stersund, for the useful comments, remarks and also for assistance with the
proof of Proposition \ref{P2,2}. She is also immensely grateful to Dr. Lotta
Flod\'{e}n and Dr. Marianne Olsson Lindberg, Mid Sweden University, \"{O}%
stersund, for theirs insightful comments on an earlier version of the paper.
\end{acknowledgement}

\end{document}